\documentclass[12pt]{article}
\usepackage[T1]{fontenc}
\usepackage{amssymb}
\usepackage{amsmath}
\usepackage{amsfonts}
\usepackage{amsthm}
\usepackage[showonlyrefs]{mathtools}
\mathtoolsset{showonlyrefs}
\numberwithin{equation}{section}
\newtheorem{theorem}{Theorem}[section]
\newtheorem*{theorema}{Theorem A}

\newtheorem{lemma}{Lemma}[section]

\theoremstyle{remark}

\newtheorem*{ack}{Acknowledgment}
\def\F{\mathcal{F}}
\def\R{\mathbb{R}}

\def\C{\mathbb{C}}

\def\N{\mathbb{N}}
\def\D{\mathbb{D}}

\def\re{\operatorname{Re}}
\def\im{\operatorname{Im}}
\begin{document}
\title{Radially distributed values and normal families}
\author{Walter Bergweiler and Alexandre Eremenko\thanks{Supported by NSF grant DMS-1665115}}
\date{}
\maketitle
\begin{abstract}
Let $L_0$ and $L_1$ be two distinct rays emanating from the origin and let 
${\mathcal F}$ be the family of all functions holomorphic in the unit disk ${\mathbb D}$
for which all zeros lie on $L_0$ while all $1$-points lie on $L_1$.
It is shown that ${\mathcal F}$ is normal in~${\mathbb D}\backslash\{0\}$.
The case where $L_0$ is the positive real axis and $L_1$ is the negative
real axis is studied in more detail.
\end{abstract}
\section{Introduction and results} \label{intro}
A major guideline in the theory of normal families is the heuristic Bloch principle
which says that the family of all holomorphic functions having a certain property 
is likely to be normal if all entire functions with this property are constant.
The classical example is the property to omit the values $0$ and $1$, in which
case the statement about normal families is Montel's theorem while the statement about
entire functions is Picard's theorem. 
For a thorough discussion of Bloch's principle, including numerous further examples 
(and counter-examples), we refer to~\cite{Bergweiler2006}, \cite{St}
and~\cite{Zalcman1998}.

There is a considerable literature on entire (and meromorphic) 
functions with radially distributed zeros and $1$-points.
In contrast, the question whether the radial distribution of zeros and $1$-points relates to 
normality does not seem to have been studied yet. In this paper we obtain results of this
type.

First we mention some results about entire functions with radially distributed values.
A classical theorem of Edrei~\cite{Edrei1955} says
 that if the zeros and $1$-points of an entire function $f$ lie
on finitely many rays, and if $\omega$ is the smallest angle between these rays,
then the order of $f$ is at most $\pi/\omega$.
Together with results of Biernacki \cite[p.~533]{Biernacki1929} or Milloux~\cite{Milloux1927}
this yields the following.
\begin{theorema} 
There is no transcendental entire function for which all zeros lie on one ray and 
all $1$-points lie on a different ray.
\end{theorema}
It is a simple consequence of Rolle's Theorem that a (non-constant) polynomial
for which all zeros lie on one ray and all $1$-points lie on a different ray has degree~$1$.

The following result is a normal family analogue of Theorem~A.
Here $\D$ denotes the unit disk.
\begin{theorem}\label{thm1}
Let $L_0$ and $L_1$ be two distinct rays emanating from the origin and 
let $\F$ be the family of all functions holomorphic in $\D$ for which all zeros 
lie on $L_0$ and all $1$-points lie on $L_1$.
Then $\F$ is normal in $\D\backslash\{0\}$.
\end{theorem}
We note that Theorem~A follows from Theorem~\ref{thm1}. Indeed, let $f$ be a transcendental entire
function with all zeros on a ray $L_0$ and all $1$-points on a different ray $L_1$. Let $(z_k)$ be 
a sequence tending to $\infty$ such that $|f(z_k)|\leq 1$ and consider the functions $f_k(z)=f(2|z_k|z)$.
By Theorem~\ref{thm1}, the $f_k$ form a normal family in $\D\backslash\{0\}$. On the other hand, we have
\begin{equation}\label{0a1}
\min_{|z|=\frac12}|f_k(z)|\leq 1
\quad\text{and}\quad
\max_{|z|=\frac12}|f_k(z)|
=\max_{|z|=|z_k|}|f(z)|
\to \infty,
\end{equation}
which implies that there exists a point of modulus $\frac12$ where the $f_k$ are not normal.
This contradiction completes the proof of Theorem~A.

Functions of the form $f_k(z)=c_k(z-a_k)$ where $a_k\to 0$ and $c_k\to\infty$ show that 
the family $\mathcal{F}$ in Theorem~\ref{thm1}
is not normal at~$0$, regardless of the choice of the rays.
The following result says that 
in the case that all zeros are positive and all $1$-points are negative,
all non-normal sequences are essentially of this form.
\begin{theorem}\label{thm2}
Let $\F$ be the family of all functions holomorphic in $\D$ for which all zeros are positive
and all $1$-points are negative. Then $\F$ is normal in $\D\backslash\{0\}$
and every sequence $(f_k)$ in $\F$ which does not have a subsequence convergent in $\D$ is of 
the form
\begin{equation}\label{0a}
f_k(z)=(z-a_k)g_k(z)
\end{equation}
where $g_k\to\infty$ locally uniformly in $\D$ and $a_k\to 0$.
\end{theorem}
An important ingredient in the proofs of Theorems~\ref{thm1} and~\ref{thm2} will be the following
result.
\begin{theorem}\label{thm3}
Let $D$ be a domain and let $L$ be a straight line which divides
$D$ into two subdomains $D^+$ and $D^-$. 
Let $\F$ be a family of functions holomorphic in $D$ which do not have zeros 
in $D$ and for which all $1$-points lie on~$L$.

Suppose that $\F$ is not normal at $z_0\in D\cap L$ and let $(f_k)$ be a sequence
in $\F$ which does not have a subsequence converging in any neighborhood of~$z_0$.
Suppose that $(f_k|_{D^+})$ converges.
Then either $(f_k|_{D^+})\to 0$ and $(f_k|_{D^-})\to \infty$  or
$(f_k|_{D^+})\to \infty$ and $(f_k|_{D^-})\to 0$.
\end{theorem}
By Montel's theorem, a family $\F$ as in Theorem~\ref{thm3}
is normal in $D^+$ and $D^-$. Thus a point $z_0$ of non-normality
automatically lies on~$L$. Moreover, given any sequence $(f_k)$ in $\F$ one may
achieve that $(f_k|_{D^+})$ converges by passing to a subsequence.

A corresponding result holds for families of meromorphic functions
which omit two values and for which all preimages of a third value
lie on a straight line. This can be reduced to the situation of
Theorem~\ref{thm3} by a fractional linear transformation.

The proof of Theorem~\ref{thm3} is based on an extension of Zalcman's lemma (see
Lemmas~\ref{lemma1} and~\ref{lemma2}
below) which may be of independent interest.
Zalc\-man's lemma says that there exist $z_k\in\D$ and $\rho_k>0$ with $z_k\to z_0$ and $\rho_k\to 0$
such that, after passing to a subsequence, $f_k(z_k+\rho_kz)\to g(z)$ for some entire function~$g$.
Using that the $f_k$ have no zeros one can show that $g$ has the form $g(z)=e^{cz+d}$.
Thus $f_k(z_k+\rho_kz)$ is close to $e^{cz+d}$ in certain disks. 
Our generalization of Zalcman's lemma gives a lower bound for the size of these disks.
Moreover, we quantify how close $f_k(z_k+\rho_kz)$ and $e^{cz+d}$ are (Lemma~\ref{lemma5}).
This yields that $f_k$ is large at some point in  one  of the domains $D^+$ and $D^-$ and small at some 
point in the other one. Landau's theorem (Lemma~\ref{lemma4}) will then be used to show that
$f_k$ is large or small within the whole domain $D^+$ or $D^-$, respectively.

The methods used in the proof of Theorem~\ref{thm2} apply to another problem,
namely what restrictions the zeros and $1$-points of a holomorphic function  $f\colon\D\to\C$
must satisfy. This problem is important in control theory; see, e.g., \cite{Blondel1994} or
 \cite[Theorem~2]{Blondel1995}.
Goldberg~\cite{Goldberg1973} showed that there exists an absolute constant $A_2$ such that if the number of zeros
and $1$-points of $f$ are distinct and different from $0$, then
at least one zero or $1$-point has modulus greater than or equal to~$A_2$.
The value of the largest constant $A_2$ with this property is not known, but
the estimates $0.005874\leq A_2\leq 0.02529$
were obtained in~\cite[Theorems~1.3 and~1.4]{BE}.

In the following result the numbers of zeros and $1$-points are allowed to be equal, but we put a 
restriction on their arguments.
\begin{theorem}\label{thm4}
Let  $f\colon\D\to\C$ be holomorphic and $0<r<1$. 
Suppose that all zeros of $f$ are in $[0,r]$ while all $1$-points of $f$ are in $[-r,0]$.
Suppose also that $f$ assumes both values $0$ and $1$ at least once, and assumes one of these 
values at least twice.
Then $r\geq C$ for some absolute constant $C$. In fact, this holds for
$C=0.000024$.
\end{theorem}
The value of $C$ given in this theorem is certainly not best possible.
On the other hand, the example in~\cite[\S\S 6--7]{BE} showing that $A_2\leq 0.02529$
also yields that the best possible constant $C$ in Theorem~\ref{thm4} satisfies $C\leq 0.02529$.

In the above theorems we study the case that the zeros lie on one ray 
and the $1$-points lie on a different ray.
Entire functions for which the zeros lie on a finite system of rays and the 
$1$-points lie on another finite system of rays where studied in~\cite{BEH}.
For example, it was shown in~\cite[Theorem~2]{BEH} that
if $f$ is a transcendental entire function with infinitely many zeros and $1$-points such that 
the zeros lie on a ray $L_0$ and while the  $1$-points
of $f$ lie on the union of two rays $L_1$ and $L_{-1}$, each of which is distinct from $L_0$,
then $\angle(L_0,L_1)=\angle(L_0,L_{-1})<\pi/2$.
On the other hand, examples of such functions $f$ with 
$\angle(L_0,L_1)=\angle(L_0,L_{-1})=\alpha$ were constructed
in~\cite[Theorem~3]{BEH} for $\alpha$ of the form $\alpha=2\pi/n$ with $n\in\N$ and
in~\cite{Eremenko2015} for any $\alpha\in(0,\pi/3]$.

It is conceivable that our theorems have generalizations to the case
that the zeros and $1$-points are distributed on several rays.

\section{Lemmas} \label{lemmas}
The following result due to Zalcman~\cite{Zalcman1975} has turned out to be very useful tool in the theory of 
normal families -- and in particular in the exploration of Bloch's principle.
\begin{lemma}\label{lemma1}
{\bf (Zalcman's Lemma)}
Let $D\subset\C$ be a domain, $\mathcal{F}$ a family of functions meromorphic in $D$ and $z_0\in D$.
Then $\mathcal{F}$ is not normal at $z_0$ if and only if 
there exist a sequence $(f_k)$ in $\mathcal{F}$, a sequence $(z_k)$ in $D$, a sequence $(\varrho_k)$ of positive real numbers and a non-constant function $g$ meromorphic in $\C$ such that $z_k\to z_0$, $\varrho_k\to 0$ and
\begin{equation}\label{1a}
f_k(z_k+\varrho_kz)\to g(z)
\end{equation}
locally uniformly in $\C$ with respect to the spherical metric.
Moreover, we have $g^\#(z)\leq 1=g^\#(0)$ for all $z\in\C$.
\end{lemma}
Let
\begin{equation}\label{1a1}
f^\#(z)=\frac{|f'(z)|}{1+|f(z)|^2}
\end{equation}
be the spherical derivative of a meromorphic function~$f$.
The proof of Zalc\-man's lemma (besides~\cite{Zalcman1975} see also~\cite[Section~4]{Bergweiler1998} 
or~\cite[p.~217f]{Zalcman1998})
proceeds by showing that for suitably chosen $f_k$, $z_k$ and $\varrho_k$ there exists a sequence $(R_k)$ tending to $\infty$ such that
\begin{equation}\label{1b}
g_k(z)=f_k(z_k+\varrho_kz)
\end{equation}
is defined in the disk $D(0,R_k)$ and satisfies $g_k^\#(0)=1$ as well as
\begin{equation}\label{1c}
g_k^\#(z)\leq 1+o(1)\quad \text{for }|z|\leq R_k\text{ as }k\to\infty.
\end{equation}
Marty's theorem then implies that the $g_k$ form a normal family. Passing to a subsequence one may now achieve \eqref{1a}.

We shall need the following result which relates $R_k$ and $\varrho_k$ 
and quantifies the $o(1)$-term in~\eqref{1c}.
Here and in the following $D(a,r)$ and $\overline{D}(a,r)$ denote
the open and closed disk of radius $r$ centered at a point $a\in\C$.

\begin{lemma}\label{lemma2}
Let $t_0>0$ and $\varphi\colon [t_0,\infty)\to (0,\infty)$ be a non-decreasing function such that 
$\varphi(t)/t\to 0$ as $t\to\infty$ and
\begin{equation}\label{1d}
\int^\infty_{t_0}\frac{dt}{t\varphi(t)}<\infty.
\end{equation}
Then one may choose $f_k$, $z_k$ and $\varrho_k$ in Zalcman's Lemma~\ref{lemma1} such that
\begin{equation}\label{1d1}
R_k:=\frac 1 {\varrho_k\varphi\!\left(1/\varrho_k\right)}\to\infty
\end{equation}
as $k\to\infty$ and
the functions $g_k$ given by \eqref{1b} are defined in the disks $D(0,R_k)$ 
and satisfy
\begin{equation}\label{1d2}
g_k^\#(z)\leq 1+\frac{|z|}{R_k}\quad\text{for }|z|<R_k.
\end{equation}
\end{lemma}
Since in Zalcman's lemma the functions $f_k$ are considered only in small neighborhoods of
the points $z_k$, the sequences $(\rho_k)$ and $(R_k)$ occurring in~\eqref{1b} and~\eqref{1c} 
must satisfy $\rho_k R_k\to 0$. Equation~\eqref{1d1} says that the sequence $(\rho_k R_k)$ tends to~$0$
slowly in some sense.

To prove Lemma~\ref{lemma2}, we first prove the following lemma.

\begin{lemma}\label{lemma3}
Let $\varphi$ be as in Lemma~\ref{lemma2}. Then for every $\varepsilon>0$ there exists $K>0$ with the following property:

If $a\in \C$ and $f\colon D(a,\varepsilon)\to \C$ is holomorphic with $f^\#(a)\geq K$, then there exists $c\in D(a,\varepsilon)$ such that with
\begin{equation}\label{1e}
\varrho=\frac 1 {f^\#(c)}\quad\text{and}\quad s=\frac{f^\#(c)}{3\varphi(f^\#(c))}
\end{equation}
the disk $D(c,\varrho s)$ is contained in $D(a,\varepsilon)$ and the function $g\colon D(0,s)\to\C$ defined by
\begin{equation}\label{1f}
g(z)=f(c+\varrho z)
\end{equation}
satisfies
\begin{equation}\label{1g}
g^\#(z)\leq \frac 1  {\displaystyle 1-\frac{|z|}{s}}.
\end{equation}
\end{lemma}
\begin{proof}
Let $K\geq t_0$ and let $f\colon D(a,\varepsilon)\to\C$ be holomorphic with $f^\#(a)\geq K$. For $0\leq r<\varepsilon$ we put
\begin{equation}\label{1h}
H(r)=\max_{|z-a|\leq r}f^\#(z).
\end{equation}
Then $H(r)\geq f^\#(a)\geq K\geq t_0$ for all~$r$.
We claim that if $K$ is sufficiently large, with a bound depending only on
$\varphi$, $t_0$ and~$\varepsilon$, then there exists $r\in[0,\varepsilon/2)$ with
\begin{equation}\label{1i1}
r+\frac 1 {\varphi(H(r))}<\varepsilon
\end{equation}
and
\begin{equation}\label{1i}
H\!\left(r+\frac 1 {\varphi(H(r))}\right)\leq e\,H(r).
\end{equation}

Suppose that this is not the case. We put $r_0=0$ and, as long as $r_{k-1}<\varepsilon/2$, define
\begin{equation}\label{1j}
r_{k}=r_{k-1}+\frac 1 {\varphi(H(r_{k-1}))}
\end{equation}
for $k\geq 1$. 
Choosing $K$ large we can achieve that 
\begin{equation}\label{1k1}
\varphi(H(r))\geq \varphi(K)>\frac{2}{\varepsilon}
\end{equation}
for all $r\in [0,\varepsilon)$ 
and thus $r_k<r_{k-1}+\varepsilon/2<\varepsilon$. It follows that 
\begin{equation}\label{1k}
H(r_k)>e\,H(r_{k-1})>\ldots >e^kH(0)\geq e^kK
\end{equation}
and thus
\begin{equation}\label{1l}
\begin{aligned}
  r_k & =\displaystyle\sum^{k}_{j=1}(r_{j}-r_{j-1}) 
 = \displaystyle\sum^{k}_{j=1} \frac 1 {\varphi(H(r_{j-1}))}
\\   & 
\leq \displaystyle\sum^{k}_{j=1} \frac 1 {\varphi(e^{j-1}K)} 
\leq \displaystyle\int^\infty_{0}\frac {du} {\varphi(e^{u-1}K)}
= \displaystyle\int^\infty_{K/e}\frac {dt} {t\varphi(t)}
\end{aligned}
\end{equation}
if $K>e t_0$.
In fact, choosing $K$ large we can achieve that $r_k<\varepsilon/2$. This shows that for such $K$ the $r_k$ can indeed be defined for all $k$. Moreover, we have $H(r_k)<H(\varepsilon/2)$, contradicting \eqref{1k} for large $k$.

Thus there exists $r\in [0,\varepsilon/2)$ such that~\eqref{1i1} 
and~\eqref{1i} hold. For such $r$ we choose $b\in\overline{D}(a,r)$ with $f^\#(b)=H(r)$ and put
\begin{equation}\label{1m}
t=\frac{1}{\varphi(f^\#(b))}=\frac{1}{\varphi(H(r))}.
\end{equation}
By~\eqref{1k1} we have $t<\varepsilon/2$. 
Thus $\overline{D}(b,t)\subset D(a,\varepsilon)$.

Next we choose $c\in \overline{D}(b,t)$ such that
\begin{equation}\label{1n}
f^\#(c)\left(1-\frac{|c-b|}{t}\right)=\displaystyle\max_{|z-b|\leq t}f^\#(z)\left(1-\frac{|z-b|}{t}\right).
\end{equation}
Then 
\begin{equation}\label{1n1}
f^\#(c)\leq H(r+t)\leq eH(r)=e f^\#(b)
\end{equation}
by \eqref{1i} and \eqref{1m}.

On the other hand, the choice of $c$ implies that
\begin{equation}\label{1o}
f^\#(c)\left(1-\frac{|c-b|}{t}\right)\geq f^\#(b).
\end{equation}
Thus
\begin{equation}\label{1p}
1-\frac{|c-b|}{t} \geq \frac{f^\#(b)}{f^\#(c)} \geq \frac{1}{e}
\end{equation}
and hence
\begin{equation}\label{1r}
|c-b|\leq \left(1-\frac 1 e \right) t\leq \frac 2 3 t.
\end{equation}
This implies that $D(c,t/3)\subset D(b,t)\subset D(a,\varepsilon)$. Moreover, \eqref{1o} yields that $f^\#(c)\geq f^\#(b)$. With $\varrho$ and $s$ defined by \eqref{1e} we now find, for $|z|<s$ and thus

\begin{equation}\label{1s}
\varrho|z|<\varrho s=\frac{1}{3\varphi(f^\#(c))}\leq \frac{1}{3\varphi(f^\#(b))}=\frac t 3,
\end{equation}
that
\begin{equation}\label{1t}
\begin{aligned}
  g^\#(z) & =\varrho f^\#(c+\varrho z) \\
              & =\displaystyle\frac{\displaystyle f^\#(c+\varrho z)\left(1-\frac{|c+\varrho z-b|}{t}\right)}{\displaystyle f^\#(c)\left(1-\frac{|c-b|}{t}\right)}\cdot \frac{\displaystyle 1-\frac{|c-b|}{t}}{\displaystyle 1-\frac{|c+\varrho z-b|}{t}} 
 \\ & 
       \leq \frac{\displaystyle 1-\frac{|c-b|}{t}} {\displaystyle 1-\frac{|c+\varrho z-b|}{t}}
       \leq \displaystyle\frac{\displaystyle 1-\frac{|c-b|}{t}} {\displaystyle 1-\frac{|c-b|}{t}-\frac{\varrho|z|}{t}}
       = \displaystyle\frac{1}{\displaystyle 1-\frac{\varrho |z|}{t-|c-b|}}  .
\end{aligned}
\end{equation}
Combining this with~\eqref{1r} we deduce that 
\begin{equation}\label{1t1}
  g^\#(z) \leq \displaystyle\frac{1}{\displaystyle 1-\frac{3\varrho |z|}{t}}
\quad\text{for}\ |z|<s.
\end{equation}
By \eqref{1s} we have $3\varrho/t\leq 1/s$ and thus the last inequality yields~\eqref{1g}.
\end{proof}

\begin{proof}[Proof of Lemma~\ref{lemma2}]
Let $D$, $\mathcal{F}$ and $z_0\in D$ be as in Zalcman's lemma. So $\mathcal{F}$ is not normal at $z_0$ and thus, by Marty's theorem, there exists a sequence $(\xi_k)$ in $D$ and a sequence $(f_k)$ in $\mathcal{F}$ such that $\xi_k\to z_0$ and $f^\#_k(\xi_k)\to \infty$. For large $k$ we may apply Lemma~\ref{lemma3} with $a=\xi_k$ and $\varepsilon=\varepsilon_k$ for some sequence $(\varepsilon_k)$ tending to 0. We choose $c,\varrho$ and $s$ according to Lemma~\ref{lemma3} and put $z_k=c$, $\varrho_k=\varrho$ and $s_k=s$. It follows that the function $g_k$ defined by \eqref{1b} satisfies $g^\#_k(0)=1$ and
\begin{equation}\label{1v}
g^\#_k(z)\leq \frac{1}{\displaystyle 1-\frac{|z|}{s_k}}\quad\text{for }|z|<s_k.
\end{equation}
Noting that $1/(1-x)\leq 1+2x$ for $0\leq x\leq 1/2$ we thus have
\begin{equation}\label{1w}
g^\#_k(z)\leq 1+\frac{2|z|}{s_k}\quad\text{for }|z|\leq \frac 1 2 s_k .
\end{equation}
Putting $R_k=s_k/2$ we obtain \eqref{1d2}. As in the proof of Zalcman's lemma we can now achieve \eqref{1a} by passing to a subsequence.
Moreover, \eqref{1e} yields that
\begin{equation}\label{1x}
R_k =\frac 1 2 s_k=\frac 1 2 s=\frac{f^\#(c)}{6\varphi(f^\#(c))}
=\frac{1}{6\varrho\varphi(1/\varrho)}=\frac{1}{6\varrho_k\varphi(1/\varrho_k)}.
\end{equation}
This is \eqref{1d1} with $\varphi$ replaced by $6\varphi$. Noting that \eqref{1d} remains valid if $\varphi$ is replaced by $\varphi/6$, this yields the conclusion.
\end{proof}

The first statement in the following lemma is known as Landau's theorem~\cite[Section~6.6]{Hayman1964}.
The second statement can be deduced from the first one, but it also follows directly 
from Montel's theorem and Marty's theorem. 
\begin{lemma}\label{lemma4}
There exists absolute constants $A$ and $B$ with the following property:
If $f\colon D(a,r)\to\C$ is holomorphic and $f(z)\neq 0$ and $f(z)\neq 1$ for all $z\in D(a,r)$, then
\begin{equation}\label{2a}
\frac{|f'(a)|}{|f(a)|\left(\big{|}\log |f(a)|\big{|}+A\right)}\leq \frac 2 r
\end{equation}
and
\begin{equation}\label{2b}
f^\#(a)\leq \frac B r.
\end{equation}
\end{lemma}
Hempel~\cite[Theorem~2]{Hempel1979} and Lai~\cite{Lai1979}
showed that the sharp constant $A$ in~\eqref{2a} is given by
\begin{equation}\label{2b1}
A=\frac{\Gamma(\tfrac14)^4}{4\pi^2}=4.3768796\dots.
\end{equation}

The limit function $g$ occurring in Zalcman's lemma satisfies $g^\#(0)=1$ and $g^\#(z)\leq 1$ 
for all $z\in \C$. 
If the functions in the family under consideration are holomorphic, then the limit function is entire. 
A result of Clunie and Hayman~\cite[Theorem~3]{Clunie1966} implies that an entire function with bounded 
spherical derivative has order at most~$1$. (This result can also be deduced from Lemma~\ref{lemma5}
below.)
Thus in the case of a family of holomorphic functions
the limit function $g$ in Zalcman's lemma is of order at most~$1$.

If, in addition, the functions in our family have no zeros, then this also holds for the limit function~$g$. It follows that $g$ then must be of the form $g(z)=e^{cz+d}$ with constants $c$ and~$d$.
In fact, a simple computation using $g^\#(z)\leq g^\#(0)=1$ shows that then $|c|=2$ and $\re d=0$.

The following result is a quantitative form of the above observation.

\begin{lemma}\label{lemma6}
Let $R>2^8B$, where $B$ is the constant from Lemma~\ref{lemma4}, and let
$g\colon D(0,R)\to \C$ be holomorphic with $g(z)\neq 0$ for all $z\in D(0,R)$. Suppose further that $g^\#(0)=1$ and
\begin{equation}\label{2c}
g^\#(z)\leq 1+\frac{|z|}{R}\quad\text{for }|z|<R.
\end{equation}
Then there exists $b\in \overline{D}(0,B)$ such that $g$ has the form
\begin{equation}\label{2d}
g(z)=\exp\!\left(c(z-b)+\delta(z)\right)
\end{equation}
where $c\in \C$ with
\begin{equation}\label{2e}
1\leq 2-2^{8}\frac B R \leq|c|\leq2+2\frac{B}{R}<3
\end{equation}
and
\begin{equation}\label{2f}
|\delta(z)|\leq 2^{7}\frac{|z-b|^2}{R}\quad\text{for }|z-b|\leq\frac{1}{16}R.
\end{equation}
\end{lemma}
Corresponding to $\re d=0$ in the equation $g(z)=e^{cz+d}$ noted above one can also prove that 
$\re cb\leq C/R$ for some constant~$C$, but we will not need this result.

To prove Lemma~\ref{lemma6} we will use the following result~\cite[Lemma~3.1]{Bergweiler2016}.
\begin{lemma}\label{lemma5}
Let $f\colon D(a,r)\to\C$  be holomorphic and $K,L>0$.
Suppose that $|f(a)|\leq K$ and that $|f'(z)|\leq L$ whenever $|f(z)|= K$. Then
\begin{equation}\label{la1a}
|f(z)|< K \exp\!\left( \frac{2L}{K}|z-a|\right)
\quad\text{for}\ z\in D\!\left(a,\frac{r}{2}\right).
\end{equation}
\end{lemma}

\begin{proof}[Proof of Lemma~\ref{lemma6}]
By Lemma~\ref{lemma4} and since $g(z)\neq 0$ for all $z\in D(0,R)$ there exists
$b\in \overline{D}(0,B)$ with $g(b)=1$. Since $R>2B$ and thus $\overline{D}(b,R/2)\subset D(0,R)$ we have
\begin{equation}\label{2g}
g^\#(z)\leq 2\quad\text{for }|z-b|\leq \frac 1 2 R
\end{equation}
by \eqref{2c}. It follows that if $z\in \overline{D}(b,R/2)$ satisfies $|f(z)|=1$, then $|f'(z)|\leq 2$. Lemma~\ref{lemma5} now yields that
\begin{equation}\label{2h}
|g(z)|\leq \exp(4|z-b|)\quad\text{for }|z-b|\leq \frac 1 4 R.
\end{equation}

Since $g$ has no zeros it is of the form $g(z)=e^{h(z)}$ for some holomorphic function $h\colon D(0,R)\to \C$. 
Since $g(b)=1$ we may choose $h$ with $h(b)=0$. The equation \eqref{2h} now takes the form
\begin{equation}\label{2i}
\re h(z)=\log |g(z)|\leq 4|z-b|\quad\text{for } |z-b|\leq \frac 1 4 R.
\end{equation}
The Schwarz integral formula says that
\begin{equation}\label{schwarz0}
h(z)=\frac{1}{2\pi i} \int_{|\zeta-b|=\frac 1 4 R}\frac{\zeta+z}{\zeta-z}\re h(\zeta)\frac{d\zeta}{\zeta} +i\im h(b)
 \quad \text{for }|z-b|<\frac 1 4 R.
\end{equation}
Differentiating this we obtain
\begin{equation}\label{2j}
h'(z)=\frac{1}{\pi i} \int_{|\zeta-b|=\frac 1 4 R}\frac{\re h(\zeta)}{(\zeta-z)^2}d\zeta\quad\text{for } |z-b|<\frac 1 4 R
\end{equation}
and thus
\begin{equation}\label{2k}
\begin{aligned}
|h'(z)|
&\leq\frac{R}{2\left(\frac 1 4 R -|z-b|\right)^2}\max_{|\zeta-b|=\frac 1 4 R}\re h(\zeta)
\\ &
\leq \frac{2R|z-b|}{\left(\frac 1 4 R -|z-b|\right)^2} \quad \text{for }|z-b|<\frac 1 4 R
\end{aligned}
\end{equation}
by \eqref{2i}. Hence
\begin{equation}\label{2l}
|h'(z)|\leq 16\quad\text{for } |z-b|\leq \frac 1 8 R.
\end{equation}
This implies that
\begin{equation}\label{2m}
\begin{aligned}
|h'(z)-h'(b)|&=\left|\frac{1}{2\pi i}\int_{|\zeta-b|=\frac 1 8 R}\left(\frac{h'(\zeta)}{\zeta-z}-\frac{h'(\zeta)}{\zeta-b}\right)d\zeta\right|\\
             &=\frac{1}{2\pi}\left|\int_{|\zeta-b|=\frac 1 8 R}  \frac{h'(\zeta)(z-b)}{(\zeta-z)(\zeta-b)}d\zeta\right|\\
             &\leq \frac 1 8 R  \cdot \frac{16 |z-b|}{\left(\frac 1 8 R-|z-b|\right)\frac 1 8 R}\\
             &=\frac{16|z-b|}{\frac 1 8 R-|z-b|}\quad\text{for }|z-b|<\frac 1 8 R
\end{aligned}
\end{equation}
and hence that
\begin{equation}\label{2n}
|h'(z)-h'(b)|\leq 16^{2}\frac{|z-b|}{R}\quad\text{for }|z-b|\leq\frac {1}{16}R.
\end{equation}
Integrating this we obtain, using $h(b)=0$, that
\begin{equation}\label{2o}
\begin{aligned}
|h(z)-h'(b)(z-b)|&=\left|\int^z_b(h'(\zeta)-h'(b))d\zeta\right|\\
                 &\leq\frac{16^2}{R}\int^{|z-b|}_0 t\;dt=2^{7}\frac{|z-b|^2}{R}\quad\text{for }|z-b|\leq \frac{1}{16}R.
\end{aligned}
\end{equation}
With $c=h'(b)$ and
\begin{equation}\label{2p}
\delta(z)=h(z)-h'(b)(z-b)=h(z)-c(z-b)
\end{equation}
we thus have \eqref{2d} with $\delta(z)$ satisfying \eqref{2f}.

It remains to prove \eqref{2e}. In order to do so we note that $1+x^2\geq 2x$ for $x\in \R$ and
$h'(z)=g'(z)/g(z)$ so that
\begin{equation}\label{2q}
1=g^\#(0)=\frac{|g'(0)|}{1+|g(0)|^2}=\frac{|g(0)|}{1+|g(0)|^2}|h'(0)|\leq \frac 1 2 |h'(0)|.
\end{equation}
Thus $|h'(0)|\geq 2$ and hence \eqref{2n} yields that
\begin{equation}\label{2r}
|c|=|h'(b)|\geq|h'(0)|-|h'(0)-h'(b)|
   \geq 2-2^{8}\frac{|b|}{R}\geq 2-2^{8}\frac B R.
\end{equation}
On the other hand, since $g(b)=1$, we have
\begin{equation}\label{2s}
\begin{aligned}
|c|&=|h'(b)|=2\frac{|g(b)|}{1+|g(b)|^2}|h'(b)|
\\ &
=2g^\#(b)
\leq 2\left(1+\frac{|b|}{R}\right)\leq2+2\frac{B}{R}
\end{aligned}
\end{equation}
by \eqref{2c}. Combining \eqref{2r} and \eqref{2s} we obtain \eqref{2e}.
\end{proof}

\section{Proof of Theorem~\ref{thm3}} \label{proofthm3}
We may assume without loss of generality that $L=\R$,
\begin{equation}\label{3a0}
D^+=\{z\in D\colon \im z >0\}
\quad\text{and}\quad
D^-=\{z\in D\colon \im z <0\}.
\end{equation}
We apply Lemma~\ref{lemma2} with
$\varphi(t)=(\log t)^2$. Let $f_k$, $z_k$, $\varrho_k$, $R_k$ and $g_k$ be as there.
Thus \eqref{1a}, \eqref{1b}, \eqref{1d1} and \eqref{1d2} hold.
Moreover,
\begin{equation}\label{3a1}
g_k^\#(0)=\varrho_k f_k^\#(z_k)=1.
\end{equation}

We now apply Lemma~\ref{lemma6} with $g=g_k$ and $R=R_k$. With $b$, $c$ and $\delta(z)$ as there we 
put $b_k=b$, $c_k=c$ and $\delta_k(z)=\delta(z)$.

Lemmas~\ref{lemma2} and~\ref{lemma6} describe the behavior of $f_k$ in the disk $D(z_k,\rho_k R_k)$.
We will see that $|\im z_k|=o(\rho_k R_k)$ so that for large $k$
this disk intersects both $D^+$ and~$D^-$.
Moreover, we will see that 
\begin{equation}\label{3i1}
\arg c_k =\frac{\pi}{2}+o(1)
\quad\text{or}\quad
\arg c_k =-\frac{\pi}{2}+o(1).
\end{equation}
Assuming that the second alternative in \eqref{3i1} holds we can
deduce from Lemma~\ref{lemma6} that $|f_k|$ is large
at certain points of $D(z_k,\rho_k R_k)\cap D^+$ and small at 
certain points of $D(z_k,\rho_k R_k)\cap D^-$.
Landau's theorem (Lemma~\ref{lemma4}) will then imply that $|f_k|$ is large in the whole domain $D^+$
while $|f_k|$ is small in~$D^-$.

To carry out this argument, we need explicit estimates. The relation between $\rho_k$ and $R_k$ that is
required is already given by Lemma~\ref{lemma2}. We will now use Lemma~\ref{lemma6} to obtain a 
quantitative version of~\eqref{3i1}.

In order to do so, we note that $f_k(z_k+\varrho_k b_k)=g_k(b_k)=1$ and thus
\begin{equation}\label{3a}
a_k:=z_k+\varrho_k b_k\in L=\R.
\end{equation}
Hence $|\im z_k|=O(\rho_k)=o(\rho_k R_k)$ as mentioned above.
Let 
\begin{equation}\label{3b}
b_k'=b_k+\frac{2\pi i}{c_k} 
\quad\text{and}\quad
\varepsilon_k=\frac{2^{15}}{R_k}.
\end{equation}
For $|z-b_k'|=\varepsilon_k$ and large $k$ we then have
\begin{equation}\label{3c}
\begin{aligned}
& \quad \; \left| g_k(z)-1-\left(\exp\!\left(c_k(z-b_k')\right)-1\right)\right|\\
&=
\left|  \exp\!\left(c_k(z-b_k')+\delta_k(z)\right) -\exp\!\left(c_k(z-b_k')\right)\right|\\
&=
\left|  \exp(c_k(z-b_k'))\right|\cdot\left| \exp(\delta_k(z)) -1\right|.
\end{aligned}
\end{equation}
Noting that $e^x-1\leq 2x$ for small positive $x$ we deduce from~\eqref{2e} and~\eqref{2f} that
if $|z-b_k'|=\varepsilon_k$ and $k$ is large, then 
\begin{equation}\label{3c1}
\begin{aligned}
& \quad\; \left| g_k(z)-1-\left(\exp\!\left(c_k(z-b_k')\right)-1\right)\right|
\\ &
\leq
2 \exp(|c_k|\varepsilon_k)\cdot \left|\delta_k(z)\right|
\leq
2^{8} \exp (3\varepsilon_k)\frac{\left|z-b_k'\right|^2}{R_k}
\\ &
\leq
2^{9} \frac{\left(\left|b_k-b_k'\right|+\varepsilon_k\right)^2}{R_k}
\leq
2^{9} \left(\frac{2\pi}{|c_k|}+1\right)^2 \frac{1}{R_k}
\leq
2^{15} \frac{1}{R_k}=\varepsilon_k.
\end{aligned}
\end{equation}
On the other hand, for $|z-b_k'|=\varepsilon_k$ and large $k$ we also have
\begin{equation}\label{3d}
\left|  \exp(c_k(z-b_k'))-1\right|
\geq \frac34 \left|  c_k(z-b_k')\right|
> \left|  z-b_k'\right| =\varepsilon_k.
\end{equation}
It now follows from Rouch\'e's theorem that there exists $b_k^*\in D(b_k',\varepsilon_k)$
such that $g_k(b_k^*)=1$.
With
\begin{equation}\label{3a2}
a_k^*:=z_k+\varrho_k b_k^*
\end{equation}
we thus have  $f_k(a_k^*)=1$ and hence $a_k^*\in\R$.
Together with~\eqref{3a} we find that
\begin{equation}\label{3e}
a_k^*-a_k=\varrho_k(b_k^*-b_k)\in\R.
\end{equation}
It follows that
\begin{equation}\label{3f}
\left|\im (b_k'-b_k)\right|
=\left|\im (b_k'-b_k^*)\right|
\leq\left|b_k'-b_k^*\right|< \varepsilon_k
\end{equation}
while 
\begin{equation}\label{3g}
\left|b_k'-b_k\right| =\frac{2\pi}{|c_k|}\geq \frac{2\pi}{3}.
\end{equation}
Hence 
\begin{equation}\label{3h}
\arg (b_k'-b_k) =\arg\!\left(\frac{2\pi i}{c_k}\right) =O( \varepsilon_k) 
\end{equation}
which implies that 
\begin{equation}\label{3i}
\arg c_k =\frac{\pi}{2}+O( \varepsilon_k) 
\quad\text{or}\quad
\arg c_k =-\frac{\pi}{2}+O( \varepsilon_k)
\end{equation}
as $k\to\infty$.
We assume first that the second alternative in \eqref{3i} holds for all $k$ so that
\begin{equation}\label{3j}
c_k=-2i +O( \varepsilon_k)
\end{equation}
by~\eqref{2e}.
(We will see later that this corresponds to the case $f_k|_{D^+}\to\infty$.)
In view of~\eqref{3b} we thus have
\begin{equation}\label{3k}
|c_k+2i|\leq\frac{C}{R_k}
\end{equation}
for some constant~$C$.

We now fix a small positive constant $\eta$ and put
\begin{equation}\label{3l}
u_k=b_k+i \eta R_k.
\end{equation}
By Lemma~\ref{lemma6} we then have
\begin{equation}\label{3m}
g_k(u_k)=\exp(c_k(u_k-b_k)+\delta_k(u_k)) =\exp(c_ki\eta R_k+\delta_k(u_k))
\end{equation}
and thus \eqref{2f} and \eqref{3k} yield that
\begin{equation}\label{3n}
\begin{aligned}
\log |g_k(u_k)|
&=\re ( c_ki\eta R_k+\delta_k(u_k))
\\ &
= \re ( 2\eta R_k +(c_k+2i)i\eta R_k+\delta_k(u_k))
\\ &
\geq  2\eta R_k -\eta C-2^{7}\eta^2 R_k
\geq \eta R_k
\end{aligned}
\end{equation}
for large $k$,
provided $\eta$ has been chosen small enough.
With 
\begin{equation}\label{3o}
\alpha_k:=z_k+\varrho_k u_k
\end{equation}
we thus have
\begin{equation}\label{3p}
\log |f_k(\alpha_k)|\geq \eta R_k
\end{equation} 
for large~$k$.
Next we note that \eqref{3a}, \eqref{3l} and \eqref{3o} imply that
\begin{equation}\label{3q}
\alpha_k=z_k+\varrho_k u_k =z_k+\varrho_k(b_k+i\eta R_k)
=a_k+i \eta\varrho_k R_k
\end{equation}
and hence
\begin{equation}\label{3r}
\im \alpha_k=\eta\varrho_k R_k .
\end{equation}
Since $z_k\to z_0$, $\rho_k\to 0$ and $|b_k|\leq B$ we deduce from~\eqref{1d1} that
\begin{equation}\label{3s}
\alpha_k=z_k+\varrho_k b_k+ i \eta \varrho_k R_k\to z_0
\end{equation}
as $k\to\infty$.

Let now $d>0$ be such that the straight line segment connecting $z_0$ and $z_1:=z_0+id$ is
contained in $D$.
We put $\beta_k=\alpha_k+id$ and note that
\begin{equation}\label{3t}
\beta_k\to z_0+id=z_1 \in D^+
\end{equation}
as $k\to\infty$ by~\eqref{3s}.
We may thus assume that the line segment connecting $\alpha_k$ and $\beta_k$ is contained 
in $D$ for all $k\in\N$.

Let $A$ be the constant from Lemma~\ref{lemma4}. By~\eqref{3p} we may assume that 
\begin{equation}\label{3u}
\log|f_k(\alpha_k)| >A 
\end{equation}
for all $k\in\N$. Let now 
\begin{equation}\label{3v}
t_k=\min \{ t\in [0,d]\colon \log |f_k(\alpha_k+it)|\leq A\},
\end{equation}
with $t_k=d$ if $\log |f_k(\alpha_k+it)|>A$ for all $t\in [0,d]$.
Put $\gamma_k=\alpha_k+it_k$. 
Since $f_k(z)\neq 0,1$ for $\im z>0$, Lemma~\ref{lemma4} and~\eqref{3r} yield that 
if $0\leq t\leq t_k$, then
\begin{equation}\label{3w}
\begin{aligned}
\frac{|f_k'(\alpha_k+it)|}{|f_k(\alpha_k+it)|\log |f_k(\alpha_k+it)|}
& \leq
\frac{2 |f_k'(\alpha_k+it)|}{|f_k(\alpha_k+it)|(\log |f_k(\alpha_k+it)|+A)}
\\ &
\leq 
\frac{4}{\im(\alpha_k+it)}
= 
\frac{4}{\eta \varrho_k R_k+t} .
\end{aligned}
\end{equation}
For suitable branches of the logarithm we thus obtain
\begin{equation}\label{3x}
\begin{aligned}
|\log\log f_k(\alpha_k)-\log\log f_k(\gamma_k)|
&= 
\left|\int_0^{t_k} \frac{f_k'(\alpha_k+it)}{f_k(\alpha_k+it)\log f_k(\alpha_k+it)} dt\right|
\\ &
\leq 
\int_0^{t_k} \frac{|f_k'(\alpha_k+it)|}{|f_k(\alpha_k+it)|\log |f_k(\alpha_k+it)|} dt
\\ &
\leq 
4 \int_0^{t_k} \frac{dt}{\eta \varrho_k R_k+t} 
= 4\log\!\left( 1+\frac{t_k}{\eta\varrho_k R_k}\right)
\\ &
\leq 4\log\!\left( 1+\frac{d}{\eta\varrho_k R_k}\right).
\end{aligned}
\end{equation}
On the other hand,
\begin{equation}\label{3y}
\begin{aligned}
|\log\log f_k(\alpha_k)-\log\log f_k(\gamma_k)|
& \geq 
\log|\log f_k(\alpha_k)|-\log|\log f_k(\gamma_k)|
\\ &
\geq 
\log\log|f_k(\alpha_k)|-\log(\log|f_k(\gamma_k)|+\pi)
\end{aligned}
\end{equation}
if the branch of the logarithm is suitably chosen.
Combining the last two inequalities we deduce that
\begin{equation}\label{3z}
\log(\log|f_k(\gamma_k)|+\pi) \geq \log\log |f_k(\alpha_k)|
 -4\log\!\left( 1+\frac{d}{\eta\varrho_k R_k}\right).
\end{equation}
By \eqref{1d1} and our choice of $\varphi$ we have 
\begin{equation}\label{4a}
\frac{1}{\varrho_k R_k} =\varphi\!\left(\frac{1}{\varrho_k}\right)=\left(\log \frac{1}{\varrho_k}\right)^2 .
\end{equation}
Combining this with \eqref{3p} and \eqref{3z}
we obtain 
\begin{equation}\label{4c}
\log(\log|f_k(\gamma_k)|+\pi) \geq \log(\eta R_k) - 4\log\!\left( 1+\frac{d}{\eta} \left(\log \frac{1}{\varrho_k}\right)^2\right).
\end{equation}

It follows from~\eqref{4a} that 
\begin{equation}\label{4b}
\log \frac{1}{\varrho_k}\sim \log R_k 
\end{equation}
as $k\to\infty$.
Inserting this in the previous equation
we can now deduce that 
\begin{equation}\label{4d}
\log(\log|f_k(\gamma_k)|+\pi) \geq (1-o(1)) \log R_k
\end{equation}
as $k\to\infty$.
First this yields that $\gamma_k=\beta_k$ for large $k$ since otherwise we have
$\log|f_k(\gamma_k)|=A$. Hence
\begin{equation}\label{4e}
\log\log|f_k(\beta_k)| \geq (1-o(1)) \log R_k
\end{equation}
and thus $|f_k(\beta_k)|\to\infty$ as $k\to\infty$.
Since $\beta_k\to z_1\in D^+$ we deduce that $f_k|_{D^+}\to\infty$.
Essentially the same argument yields that $f_k|_{D^-}\to 0$.

Had we assumed that the first alternative holds in~\eqref{3i}, we would
have obtained $f_k|_{D^-}\to\infty$ and $f_k|_{D^+}\to 0$.
Our hypothesis that $f_k|_{D^+}$ converges thus implies that the same alternative in~\eqref{3i} holds 
for all large~$k$. This completes the proof.

\section{Proof of Theorem~\ref{thm1}} \label{proofthm1}
The rays $L_0$ and $L_1$ divide $\D$ into two subdomains $D_1$ and $D_2$.
By Montel's theorem, $\F$ is normal in $D_1\cup D_2$.
Suppose that $\F$ is not normal at some point $z_1\in \D\backslash\{0\}$.
Then $z_1\in (L_0\cup L_1)\backslash \{0\}$.
Without loss of generality we may assume that $z_1\in L_1\backslash\{0\}$, since otherwise
we may consider the family $\{1-f\colon f\in\F\}$ instead of $\F$, which corresponds to 
interchanging the roles of $L_0$ and $L_1$.

Theorem~\ref{thm3} yields that there exists a sequence
$(f_k)$ in $\F$ which tends
to $\infty$ in one of the domains $D_1$ and $D_2$
and which tends to $0$ in the other one.
It follows that $(f_k)$ is not normal at any point of $L_0\cup L_1$.
Applying Theorem~\ref{thm3} to  the family $\{1-f\colon f\in\F\}$ with
some point $z_0\in L_0\backslash\{0\}$
we see that a subsequence of $(f_k)$ tends
to $\infty$ in one of the domains $D_1$ and $D_2$
and to $1$ in the other one.
This is a contradiction.

\section{Proof of Theorem~\ref{thm2}} \label{proofthm2}
The proof of Theorem~\ref{thm2} will combine Theorem~\ref{thm1} with the following lemma.
\begin{lemma}\label{lemma7}
Let $f\colon\D\to\C$ be holomorphic, with all zeros positive and all $1$-points 
negative. Suppose that there exists $r\in (0,1)$ such that
\begin{equation}\label{5a}
\min_{|z|=r}|f(z)|>1.
\end{equation}
Then $f$ has either has no zeros and no $1$-points in $D(0,r)$, or exactly one zero and one $1$-point
in $D(0,r)$, both of which are simple.
\end{lemma}
\begin{proof}
Rouch\'e's theorem implies that, counting multiplicities,
$f$ has the same number of zeros and $1$-points in $D(0,r)$.
Denote this number by~$n$. We assume that $n\neq 0$ and thus have to show that $n=1$.
We consider the function $g\colon\D\to\C$,
\begin{equation}\label{5b}
g(z)=f(z)\overline{f(\overline{z})}.
\end{equation}
Then $g(z)\in\R$ for $z\in\R$.
Counting multiplicities, $g$ has $2n$ zeros, all of which are non-negative and of even multiplicity,
and the number of $1$-points of $g$ in the interval $(-r,0)$ is at least~$n$.
Moreover,
\begin{equation}\label{5c}
\varrho:=\min_{|z|=r}|g(z)|>1.
\end{equation}
Let $U$ be the component of $g^{-1}(D(0,\varrho))$ which contains the leftmost $1$-point of~$g$
in $(-r,0)$. Then $U$ is simply-connected, $U\subset D(0,r)$, and the map $g\colon U\to D(0,\varrho)$
is proper. In particular, counting multiplicities, $U$ contains the same number of zeros and
$1$-points of $g$. Denote this number by $m$.
Since $U$ is simply-connected and symmetric with respect to the real axis 
and since $U$ contains the leftmost $1$-point of $g$ in $(-r,0)$
and at least one zero, $U$ actually contains all $1$-points in $(-r,0)$. 
The Riemann-Hurwitz formula yields that, counting multiplicities, $g$ has $m-1$ critical points 
in~$U$. 

Let $x_1,\dots,x_k$ be the zeros of  $f$ and hence of $g$ that are contained in $U$,
ordered such that $0<x_1<x_2<\dots<x_k$.
Denote by $\mu_j$ the multiplicity of $x_j$ with respect to~$g$. Thus the $\mu_j$ are even and
\begin{equation}\label{5d}
\sum_{j=1}^k \mu_j= m .
\end{equation}
The $x_j$ are also critical points of $g$ of multiplicity $\mu_j-1$.
Moreover, 
Rolle's theorem yields that each interval $(x_j,x_{j+1})$ also contains a critical point,
for $1\leq j\leq k-1$. Altogether the number of critical points of $g$ in the interval
$(0,r)$ is thus at least
\begin{equation}\label{5e}
(k-1)+ \sum_{j=1}^k (\mu_j-1) = m-1 .
\end{equation}
We conclude that all critical points of $g$ in $U$ are contained in the interval $(0,r)$.
However, if $n\geq 2$, then the interval $(-r,0)$ contains at least two $1$-points of $f$ and hence 
at least two $1$-points of~$g$.
Using  Rolle's theorem again we see $g$ has a critical point between two $1$-points and  thus
a critical point in $(-r,0)\cap U$. This is a contradiction. Thus $n=1$ as claimed.
\end{proof}
\begin{proof}[Completion of the proof of Theorem~\ref{thm2}]
Theorem~\ref{thm1} yields that $\F$ is normal in $\D\backslash\{0\}$.
Suppose that some subsequence of $(f_k)$ tends to a finite limit function, 
say $f_{k_j}\to g$ locally uniformly in $\D\backslash\{0\}$.
Then $g$ is holomorphic in $\D\backslash\{0\}$. For $r\in (0,1)$
let $M(r,g)=\max_{|z|=r}|g(z)|$ denote the maximum modulus of~$g$.
Then
\begin{equation}\label{5f}
M\!\left(r,f_{k_j}\right)
\leq M\!\left(r,g\right)+1
\end{equation}
for large~$j$. This implies that the $f_{k_j}$ form a normal family 
in $D(0,r)$, contradicting our hypothesis that no subsequence of
$(f_k)$ is normal in~$\D$. Hence $f_k\to\infty$ in $\D\backslash\{0\}$.

Next we claim that  for large $k$ the function $f_k$ has a zero in $D(0,r)$. 
Indeed, otherwise we would have 
\begin{equation}\label{5g}
\min_{|z|\leq r}|f_{k_j}(z)|=\min_{|z|=r}|f_{k_j}(z)|\to\infty
\end{equation}
as $j\to\infty$ for some subsequence $(f_{k_j})$, implying that this subsequence is normal at $0$,
contradicting our hypothesis.

For large $k$ we thus find that $\min_{|z|=r}|f_k(z)|>1$ and that $f_k$ has a zero in $D(0,r)$.
Lemma~\ref{lemma7} yields that $f_k$ has exactly one zero $a_k$ in $D(0,r)$.
Hence for large $k$ the function
 $f_k$ has the form~\eqref{0a} with some function $g_k$ which is holomorphic in~$\D$ 
and has no zeros in $D(0,r)$.
Since $f_k\to\infty$ in $\D\backslash\{0\}$ we find that $g_k\to\infty$ in~$\D$.
Moreover, since $0$ is the only point where the $f_k$ are not normal we conclude
that $a_k\to 0$.
\end{proof}

\section{Proof of Theorem~\ref{thm4}} \label{proofthm4}
In the proof of the following lemma, Landau's theorem (Lemma~\ref{lemma4}) is 
applied in
a similar way as in the proof of Theorem~\ref{thm3}.
However, this time we will use it with the sharp constant $A$ given by~\eqref{2b1}.

\begin{lemma}\label{lemma8}
Let $f\colon \D\to \C$ be holomorphic and $C =0.000024$.
Suppose that all zeros and $1$-points are contained in $D(0,C)$ and that $f$ has at
least one $1$-point and at least two zeros.
Then
\begin{equation}\label{b}
\min_{|z|=\sqrt{C}}|f(z)|>1
\end{equation}
\end{lemma}

\begin{proof}

Let $b$ be a $1$-point of $f$ and let $a_1,\cdots,a_m$ be the zeros of $f$ so that $m\geq 2$. The Poisson-Jensen formula yields that if $C<r<1$, then
\begin{equation}\label{d}
\begin{aligned}
0&=\log|f(b)|\\
   &=\frac{1}{2\pi}\int^{2\pi}_0 \log|f(re^{i\theta})|\re\!\left(\frac{re^{i\theta}+b}{re^{i\theta}-b}\right)d\theta
  -\sum^m_{j=1}\log \left|\frac{r^2-\overline{a}_jb}{r(b-a_j)}\right|\\
   &\leq \log M(r,f)-2\log \frac{r^2-C^2}{2C r}.
\end{aligned}
\end{equation}
It follows that
\begin{equation}\label{e}
\log M(\sqrt{C},f)\geq 2 \log \frac{C-C^2}{2C^{3/2}}=2\log \frac{1-C}{2\sqrt{C}}>0.
\end{equation}

Let $x_0=\frac 1 2 \log C$ and
\begin{equation}\label{f}
S=\{z\colon  2x_0<\re z < 0\}.
\end{equation}
Then $g\colon  S\to \C$, $g(z)=f(e^z)$, satisfies $g(z)\neq 0$ and $g(z)\neq 1$ for all $z\in S$. For $y_0\in \R$ and $z_0=x_0+i y_0$ the map
\begin{equation}\label{g}
\phi\colon  S\to\D,\quad \phi(z)=\frac{\exp\!\left(\displaystyle\frac{\pi i}{2x_0}(z-z_0)\right)-1}{\exp\!\left(\displaystyle\frac{\pi i}{2x_0}(z-z_0)\right)+1}
\end{equation}
is biholomorphic and satisfies
\begin{equation}\label{h}
|\phi'(z_0)|=\frac{\pi}{4|x_0|}.
\end{equation}
Applying Landau's theorem to $h=g\circ\phi^{-1}$ we find that
\begin{equation}\label{i}
|h'(0)|\leq 2|h(0)|\left(\big{|}\log|h(0)|\big{|}+A\right)
\end{equation}
and hence, since $h(0)=g(z_0)$,
\begin{equation}\label{j}
\begin{aligned}
|g'(z_0)|&=|h'(0)|\cdot |\phi'(z_0)|
\\ &
\leq 2|h(0) \left(\big{|}\log|h(0)|\big{|}+A\right)\frac{\pi}{4|x_0|}\\
&=\frac{\pi}{2|x_0|}|g(z_0)|\left(\big{|}\log|g(z_0)|\big{|}+A\right).
\end{aligned}
\end{equation}

Suppose now that \eqref{b} does not hold. Then there exists $z_1,z_2\in S$ with $\re z_1=\re z_2=x_0$ and $|\im z_1-\im z_2|\leq \pi$ such that $|g(z_1)|=1$, $|g(z_2)|=M(\sqrt{C},f)$ and $|g(z)|\geq 1$ for $z$ on the line segment connecting $z_1$ and $z_2$.

With suitable branches of the logarithm we deduce from \eqref{j}, which holds
for every $z_0\in[z_1,z_2]$, that
\begin{equation}\label{k}
\begin{aligned}
& \, |\log(\log g(z_2)+A)-\log(\log g(z_1)+A)| \\ 
=& \left|\int^{z_2}_{z_1}\frac{g'(z)}{g(z)(\log g(z)+A)}dz\right|
\leq \int^{z_2}_{z_1}\frac{|g'(z)|}{|g(z)|(\log |g(z)|+A)}|dz|\\
\leq & 
\, \frac{\pi}{2|x_0|}|z_2-z_1|\leq \frac{\pi^2}{2|x_0|}
=\frac{\pi^2}{\log\frac 1 C}.
\end{aligned}
\end{equation}
On the other hand, noting that $|g(z_1)|=1$, for a suitable branch of the logarithm we have
\[\begin{aligned}
  & \,     |\log(\log g(z_2)+A)-\log(\log g(z_1)+A)| \\
 \geq & \log|\log g(z_2)+A|-\log|\log g(z_1)+A|\\
 \geq & \log(\log |g(z_2)|+A)-\log | i \arg g(z_1)+A|\\
 \geq & \log(\log M(\sqrt{C},f)+A)-\log\sqrt{A^2+\pi^2}.
\end{aligned}\]
Combining the last two estimates we obtain
\[\log(\log M(\sqrt{C},f)+A)\leq \log\sqrt{A^2+\pi^2}+\frac{\pi^2}{\log \frac 1 C}.\]
Together with \eqref{e} this yields that
\[\log\!\left(2\log \frac{1-C}{2\sqrt{C}}+A\right)\leq\log \sqrt{A^2+\pi^2}+\frac{\pi^2}{\log \frac 1 C}.\]
This condition contradicts the choice of $C$.
\end{proof}

\begin{proof}[Proof of Theorem~\ref{thm4}]
Let $f$ and $r$ be as in the theorem. We may assume that $f$ has at least two zeros,
since otherwise we can consider $1-f(-z)$ instead of $f(z)$.
Lemma~\ref{lemma7} implies that if $r<s<1$, then 
\begin{equation}\label{7a}
\min_{|z|=s}|f(z)|\leq 1.
\end{equation}
In particular, 
\begin{equation}\label{7b}
\min_{|z|=\sqrt{r}}|f(z)|\leq 1.
\end{equation}
Lemma~\ref{lemma8} now yields that $r\geq 0.000024$.
\end{proof}

\begin{ack}
We thank Xiao Yao for drawing our attention to Lai's paper~\cite{Lai1979}
and the referee who helped to improve the exposition.
\end{ack}

\noindent 
W. B.: Mathematisches Seminar\\
Christian-Albrechts-Universit\"at zu Kiel\\
Ludewig-Meyn-Str.\ 4\\
24098 Kiel\\
Germany\\

\noindent
Email: bergweiler@math.uni-kiel.de

\medskip

\noindent 
A. E.: Department of Mathematics\\
Purdue University\\
West Lafayette, IN 47907\\
USA\\

\noindent
Email: eremenko@math.purdue.edu
\end{document}